\documentclass[11pt]{article}
\usepackage{amsmath,epsfig, amsthm, amssymb, enumerate}

\usepackage{fullpage}
\usepackage{graphicx}
\usepackage{subfigure}
\usepackage{amsmath, amssymb, epsfig}
\usepackage{bm}
\usepackage{mathrsfs}
\usepackage{color}

\DeclareMathOperator*{\argmin}{arg min} 

\DeclareMathOperator*{\supp}{supp}

\newcommand{\defby}{\overset{\mathrm{\scriptscriptstyle{def}}}{=}}

\newtheorem{bigthm}{Theorem}

\newcommand{\R}{\mathbb{R}}
\newcommand{\C}{\mathbb{C}}
\newcommand{\Z}{\mathbb{Z}}
\newcommand{\N}{\mathbb{N}}

\newcommand{\vct}[1]{\bm{#1}}
\newcommand{\mtx}[1]{\bm{#1}}

 \newcommand{\scalprod}[1]{\left\langle #1  \right \rangle}
\graphicspath{{../Images/}}

\newtheorem{theorem}{Theorem}[]
\newtheorem{lemma}[theorem]{Lemma}
\newtheorem{proposition}[theorem]{Proposition}
\newtheorem{cor}[theorem]{Corollary}

\newtheorem*{mainthm}{Main Theorem}

\newtheorem{remark}[theorem]{Remark}

\newtheorem{definition}[theorem]{Definition}

\begin{document}

\title{Near-optimal compressed sensing guarantees for total variation minimization}

\author{Deanna Needell\footnote{Claremont McKenna College, 850 Columbia Ave, Claremont CA, 91711, email: dneedell@cmc.edu.}   and Rachel Ward\footnote{University of Texas at Austin, 2515 Speedway, Austin, Texas, 77842, email: rward@math.utexas.edu.  R. Ward has been supported in part by a Donald D. Harrington Faculty Fellowship, Alfred P. Sloan Research Fellowship, and DOD-Navy grant N00014-12-1-0743.}}
\maketitle
  \begin{abstract}
Consider the problem of reconstructing a multidimensional signal from an underdetermined set of measurements, as in the setting of compressed sensing.  Without any additional assumptions, this problem is ill-posed.  However, for signals such as natural images or movies, the minimal \emph{total variation} estimate consistent with the measurements often produces a good approximation to the underlying signal, even if the number of measurements is far smaller than the ambient dimensionality.  
This paper extends recent reconstruction guarantees for two-dimensional images $\mtx{x} \in \C^{N^2}$ to signals $\mtx{x} \in \C^{N^d}$ of arbitrary dimension $d \geq 2$ and to isotropic total variation problems.  To be precise, we show that a multidimensional signal $\mtx{x} \in \C^{N^d}$ can be reconstructed from ${\cal O}(s d\log(N^d))$ linear measurements ${\bm y} = \mtx{A}\mtx{x}$ using total variation minimization to within a factor of the best $s$-term approximation of its gradient.  The reconstruction guarantees we provide are necessarily optimal up to polynomial factors in the spatial dimension $d$.  
\end{abstract}

\section{Introduction}

Compressed sensing (CS) is an emerging signal processing methodology where signals are acquired in compressed form as undersampled linear measurements.    The applications of CS are abundant, ranging from radar and error correction to many areas of image processing~\cite{CSwebpage}.  The underlying assumption that makes such acquisition and reconstruction possible is that most natural signals are \textit{sparse} or \textit{compressible}.  We say that a signal $\mtx{x}\in\C^{p}$ is $s$-sparse when
\begin{equation}\label{eq:sparse}
\|\mtx{x}\|_0 \defby |\supp(\mtx{x})| \leq s \ll p.
\end{equation}
Compressible signals are those which are well-approximated by sparse signals.  
In the CS framework, we acquire $m\ll p$ nonadaptive linear measurements of the form

$$
\vct{y} = {\cal M}(\mtx{x}) + \vct{\xi},
$$
where ${\cal M}: \C^p \rightarrow \C^m$ is an appropriate linear operator and $\vct{\xi}$ is vector modeling additive noise.  The theory of CS (\cite{donoho1989uncertainty,CT05:Decoding,donoho2006compressed}) ensures that under suitable assumptions on ${\cal M}$, compressible signals can be approximately reconstructed by solving an  $\ell_1$ minimization problem:
\begin{equation}\tag{$L_1$}
\hat{\mtx{x}} = \argmin_{\vct{w}} \|\vct{w}\|_1 \quad\text{such that}\quad \|{\cal M}(\vct{w}) - \vct{y}\|_2 \leq \varepsilon;
\end{equation}
above, $\|\vct{w}\|_1 = \sum_i |w_i|$ and $\|\vct{w}\|_2 = \left(\sum_i |w_i|^2\right)^{1/2}$ denote the standard $\ell_1$ and Euclidean norms, $\varepsilon$ bounds the noise level $\|\vct{\xi}\|_2 \leq \varepsilon$, and $\argmin$ denotes the set of minimizers (we use the notation of equality rather than set notation here and use the convention that if the minimizer is not unique, we may choose a solution lexicographically; however, in this particular case under standard compressed sensing assumptions it is known that the solution is indeed unique~\cite{zhang2012necessary}).  The program $(L_1)$ may be cast as a second order cone program (SOCP) and can be solved efficiently using standard convex programming methods (see e.g.~\cite{boyd2004convex, combettes2011proximal}).

One property of the operator ${\cal M}$ that guarantees sparse signal recovery via $(L_1)$, as introduced by Cand\`es and Tao  in~\cite{CT05:Decoding}, is the restricted isometry property (RIP).
\begin{definition}\label{def:rip}
A linear operator ${\cal M} : \mathbb{C}^{p} \rightarrow \mathbb{C}^m$ is said to have the restricted isometry property (RIP) of order $s \in \mathbb{N}$ and level $\delta \in (0,1)$  if
\begin{align}
\label{rip}
(1 - \delta) \| \mtx{x} \|_2^2 \leq \| {\cal M}(\mtx{x}) \|_2^2 \leq (1 + \delta) \| \mtx{x} \|_2^2 \hspace{12mm} \textrm{ for all}\hspace{2mm} s \textrm{-sparse } \mtx{x} \in \mathbb{C}^p.
\end{align}
\end{definition}
Many distributions of random matrices of dimension $m \times p$, including randomly subsampled rows from the discrete Fourier transform~\cite{RV08:sparse} or from a bounded orthonormal system more generally~\cite{RV08:sparse,rw11,ra09-1, rw10,bdwz}, and randomly-generated circulant matrices~\cite{RRT12:circulant}, are known to generate RIP matrices of order $s$ and level $\delta \leq c < 1$ if the number of measurements  satisfies $m \approx \delta^{-2} s\log^4(p)$.  Note that here and throughout we have used the notation $u \approx v$ (analogously $u \lesssim v$) to indicate that there exists some absolute constant $C > 0$ such that $u = C v$ ($u \leq C v$).  Moreover, a matrix whose entries are independent and identical (i.i.d.) realizations of a properly-normalized subgaussian random variable will have the RIP with probability exceeding $1 - e^{-cm}$ once $m \approx \delta^{-2} s \log (p/s)$~\cite{CT04:Near-Optimal,MPJ06:Uniform,RV08:sparse,badadewa08}.

Cand\`es, Romberg, and Tao~\cite{CRT06:Stable} showed that when the measurement operator ${\cal M}$ has the RIP of order ${\cal{O}}(s)$ and sufficiently small constant $\delta$, the program $(L_1)$ recovers an estimation $\hat{\mtx{x}}$ to $\mtx{x}$ that satisfies the error bound
\begin{equation}\label{eq:L1}
\|\hat{\mtx{x}} - \mtx{x}\|_2 \leq C\left(\frac{\|\mtx{x} - \vct{x_s}\|_1}{\sqrt{s}} + \varepsilon \right),
\end{equation}
where $\vct{x_s}$ denotes the best $s$-sparse approximation to the signal $\mtx{x}$. 
Using properties about Gel'fand widths of the $\ell_1$ ball due to Kashin~\cite{Kas77:The-widths} and Garnaev--Gluskin~\cite{GG84:On-widths}, this is the optimal minimax reconstruction rate for $\ell_1$-minimization using $m \approx s \log (p/s)$ nonadaptive linear measurements.  Due to the rotational-invariance of an RIP matrix with randomized column signs \cite{kw}, a completely analogous theory holds for signals that are compressible with respect to a known orthonormal basis or tight frame $\mtx{D}$ by replacing $\vct{w}$ with $\mtx{D}^{*}\vct{w}$ inside the $\ell_1$-norm of the minimization problem $(L_1)$~\cite{CENR_Compressed,liu2012compressed}.

\subsection{Imaging with CS} 

Natural images are highly compressible with respect to their gradient representation.  
For an image $\mtx{x} \in \C^{N^2}$ one defines its discrete directional derivatives by 
\begin{eqnarray}
\mtx{x}_u : \C^{N^2} \rightarrow \C^{(N-1) \times N}, \quad \quad (\mtx{x}_u)_{j,k} &=&  \mtx{x}_{j+1,k} - \mtx{x}_{j,k}  \label{Xx} \\
\mtx{x}_v : \C^{N^2} \rightarrow \C^{N \times (N-1)}, \quad \quad
(\mtx{x}_v)_{j,k} &=& \mtx{x}_{j,k+1} - \mtx{x}_{j,k}.
  \end{eqnarray}


The discrete gradient transform $\nabla: \C^{N^2} \rightarrow \C^{N^2 \times 2}$ is defined in terms of the directional derivatives,
\begin{equation}
\label{grad2d}
\big((\nabla \mtx{x})_{j,k,1}, (\nabla \mtx{x})_{j,k,2} \big) \defby \left\{ \begin{array}{ll}
\big( (\mtx{x}_u)_{j,k}, (\mtx{x}_v)_{j,k} \big), & 1 \leq j \leq N-1, \quad 1 \leq k \leq N-1 \nonumber \\
\big( 0, (\mtx{x}_v)_{j,k} \big), & j=N, \quad 1 \leq k \leq N-1 \nonumber \\
\big((\mtx{x}_u)_{j,k}, 0 \big), & k=N, \quad 1\leq j \leq N-1 \nonumber \\
\big(0,0 \big), & j=k=N
\end{array}
\right.
\end{equation}
The \emph{anisotropic} total variation seminorm is defined as
\begin{equation}
\| \mtx{x} \|_{TV_1} \defby \sum_{j,k=1}^N \left| (\nabla \mtx{x})_{j,k,1} \right| + \left| (\nabla \mtx{x})_{j,k,2} \right|,
\end{equation}
and the \emph{isotropic} total variation seminorm as
\begin{equation}\label{eq:TV2d}
\| \mtx{x} \|_{TV_2} \defby \sum_{j,k=1}^N \big( (\nabla \mtx{x})^2_{j,k,1} + (\nabla \mtx{x})^2_{j,k,2} \big)^{1/2}.
\end{equation}
The recovery guarantees we derive apply to both anisotropic and isotropic total variation semi norms, and we use the notation $\| \mtx{x} \|_{TV}$ to refer to either choice of seminorm.  For brevity of presentation, we provide details only for the isotropic total variation seminorm, but refer the reader to~\cite{nwtrAniso} for the analysis of the anisotropic variant.

The total variation seminorm is a regularizer of choice in many image processing applications.  That is, given a collection of noisy linear measurements ${\cal M}(\mtx{x}) + \mtx{\xi}$ with $\| \mtx{\xi} \|_2 \leq \varepsilon$ of an underlying image ${\bf x} \in \C^{N^2}$, total variation minimization is used to pick from among the possibly infinitely-many images consistent with these measurements:
\begin{equation}\tag{TV}
\mtx{\hat{x}} = \argmin_{\mtx{z}} \| \mtx{z} \|_{TV}  \quad  \textrm{ such that }  \quad \| {\cal M}(\mtx{z}) -\vct{y} \|_2   \leq \varepsilon
\end{equation}
Properties of TV minimizers in inverse problems have been studied in the discrete and continuous settings~\cite{baus2013fully,nikolova2012exact,dibos1999global,rudin1992nonlinear,ring2000structural,osher2005iterative,caselles2007discontinuity,caselles2011regularity}, and convergence rates of stability measures for TV have also been established \cite{burger2004convergence,grasmair2009locally}.
In the setting of compressed sensing and more broadly in other imaging applications, total variation regularization has been used for denoising, deblurring, and inpainting  (see e.g.~\cite{CRT06:Stable,crt,candes05pr,osher2003image,chan2006total,lustig2007sparse,lustig2008compressed,liu2011total,nett2008tomosynthesis,mayin,kai2008suppression,keeling2003total,Yuan-et-al-09a,nwtv} and the references therein).  {\bf In this article, we focus on recovery guarantees for (TV) in the compressed sensing setting.  Along the way, we derive strengthened Sobolev inequalities for discrete signals lying near the null space of operators incoherent with the Haar wavelet basis, and we believe that such bounds should be useful in a broader context for understanding the connection between total variation minimization and $\ell_1$-wavelet coefficient minimization.}

While (TV) is similar to the $\ell_1$-minimization program $(L_1)$, the RIP-based theoretical guarantees for $(L_1)$ do not directly translate to recovery guarantees for (TV) because the gradient map $\mtx{z} \rightarrow \nabla \mtx{z}$ is not well-conditioned on the orthogonal complement of $\ker(\nabla$).  In fact, viewed as an invertible operator over mean-zero images, the condition number of the gradient map is proportional to the image side length $N$.\footnote{One sees that the norm of $\nabla$ is a constant whereas the norm of its inverse is proportional to $N$ (one can observe this scaling, for example, by noting it is obtained by the image whose entries are constant).}
For anisotropic $(TV)$, recovery guarantees in the compressed sensing setting were obtained in~\cite{nwtv} for two-dimensional images.  
\begin{bigthm}[from~\cite{nwtv}]\label{introThm}
For a number of measurements $m \gtrsim s\log(N^2/s)$, there are choices of linear operators ${\cal M}: \C^{N^2} \rightarrow \C^m$ for which the following holds for any image $\mtx{x} \in \C^{N^2}$: Given noisy measurements $\vct{y} = {\cal M}(\mtx{x}) + \vct{\xi}$ with noise level $\| \vct{\xi} \|_2 \leq \varepsilon$, the reconstructed image
\begin{eqnarray}
\mtx{\hat{x}} = \argmin_{\mtx{z}} \| \mtx{z} \|_{TV_1}  \quad  \textrm{such that}  \quad \| {\cal M}(\mtx{z}) - \vct{y} \|_2   \leq \varepsilon
\end{eqnarray}
satisfies the error bound
\begin{equation}
\label{stabley}
\| \mtx{x}-\mtx{\hat{x}} \|_2 \lesssim  \log(N^2/s)\Big( \frac{ \| \nabla\mtx{x} - (\nabla\mtx{x})_s \|_1}{\sqrt{s}}+ \varepsilon\Big).
\end{equation}
Here and throughout, $\vct{z}_s$ denotes the best $s$-term approximation to the array $\vct{z}$.
\end{bigthm}

In words, the total variation minimizer estimates $\mtx{x}$ to within a factor of the noise level and best $s$-term approximation error of its gradient.  The  bound in \eqref{stabley} is optimal up to the logarithmic factor $\log(N^2/s)$.

{\bf The contribution of this paper is twofold: We extend the recovery guarantees of Theorem~\ref{introThm} to the multidimensional setting $\bm{x \in \C^{N^d}}$ and to the setting of isotropic total variation minimization, for arbitrary dimension $\bm{d \geq 2}$.}   The precise statement of results is given in Theorem \ref{thm:gen}.  The proofs involve extending the Sobolev inequalities for random subspaces from~\cite{nwtv} to higher-dimensional signal structures, using bounds of Cohen, Dahmen, Daubechies, and DeVore in \cite{cddd} on the compressibility of wavelet representations in terms of the bounded variation of a function, which hold for functions in dimension $d \geq 2$.  Hence, our results for total variation, do not hold in dimension $d=1$.  See~\cite{vaiter2011robust} for results on one-dimensional total variation under assumptions other than the RIP.

\subsection{Organization}
The article is organized as follows.  In Section~\ref{sec:high} we recall relevant background material on the  multidimensional total variation seminorm and multidimensional orthonormal wavelet transform.  Section~\ref{sec:main} states our main result: total variation minimization provides stable signal recovery for signals of arbitrary dimension $d\geq 2$.  The proof of this result will occupy the remainder of the paper;  in Section~\ref{stabgrad} we prove that the signal \emph{gradient} is recovered stably, while in Section~\ref{sec:sob} we pass from stable gradient recovery to stable signal recovery using strengthened Sobolev inequalities which we derive for random subspaces.  The proofs of propositions and theorems used along the way are contained in the appendix.

\section{Preliminaries for multidimensional signal analysis}\label{sec:high}
The setting for this article is the space $\C^{N^d}$ of multidimensional arrays of complex numbers,
$$
\mtx{x} = \left(x_{\alpha} \right), \quad \alpha \equiv (\alpha_1, \alpha_2, \dots, \alpha_d) \in \{1,2, \dots, N\}^d.  
$$
From here on out, we will use the shorthand $[N]^d = \{1,2, \dots, N \}^d$.   We will also use the convention that vectors such as $\vct{x}$ (apart from index vectors $\alpha$) are boldface and their scalar components such as $x_{\alpha}$ are normal typeface.
We also treat $\C^{N^d}$ as the Hilbert space equipped with inner product 
\begin{equation}
\label{inner}
\scalprod{\mtx{x}, \mtx{y} } = \sum_{\alpha \in [N]^d} x_{\alpha} \cdot \bar{y}_{\alpha},
\end{equation}
where $\bar{{y}}$ denotes the conjugate of ${y}$.
This Hilbert space is isometric\footnote{Recall that $f \in L_2(Q)$ if $\int_{Q} | f(u) |^2 du < \infty$, and $L_2(Q)$ is a Hilbert space equipped with the inner product $\scalprod{f, g} = \int_{Q} f(u) \cdot \bar{g}(u) du$. }  to the subspace $\Sigma^d_N \subset L_2\big( [0,1)^d \big)$ of functions which are constant over cubes $[ \frac{\alpha_i - 1}{N}, \frac{\alpha_i}{N} )_{i=1}^d$ of side length $N^{-1}$, and the isometry is provided by identifying $x_{\alpha} \in \C^{N^d}$ with the function $f \in \Sigma_N^d$ satisfying $f(u) = N^{d/2} x_{\alpha}$ for $u \in [ \frac{\alpha_i - 1}{N}, \frac{\alpha_i}{N} )_{i=1}^d$. 
More generally, we denote by $\|\mtx{x}\|_p = \left(\sum_{\alpha \in [N]^d}|{x}_{\alpha}|^p\right)^{1/p}$ the entrywise $\ell_p$-norm of the signal $\mtx{x}$. 

For $\ell = 1, 2, \ldots d$, the discrete derivative of $\mtx{x}$ in the direction of $r_{\ell}$ is the array ${\vct{x}_{r_{\ell}}} \in \mathbb{C}^{N^{\ell-1}\times(N-1)\times N^{d-\ell}}$ defined component-wise by 
\begin{equation}\label{eq:multidd}
\left(\vct{x}_{r_{\ell}}\right)_{\alpha} \defby  x_{(\alpha_1, \alpha_2, \dots, \alpha_{{\ell}}+1, \dots, \alpha_d)} - x_{(\alpha_1, \alpha_2, \dots, \alpha_{\ell}, \dots \alpha_d)},
\end{equation}
and we define the $d$-dimensional discrete gradient transform $\nabla: \C^{N^d} \rightarrow \C^{d \times N^d}$ through its components
\begin{equation}
\label{grad}
\big( \nabla \mtx{x}  \big)_{\alpha} = \big( \nabla \mtx{x}  \big)_{\alpha,\ell} \defby \left\{ \begin{array}{ll}
 (x_{r_{\ell}})_{\alpha}, & \alpha_{\ell} \leq N-1,  \\
 0, & \textrm{else}
 \end{array} \right.
 \end{equation}
The $d$-dimensional \emph{anisotropic} total variation seminorm is defined as $\| \mtx{x} \|_{TV_1} \defby \|\nabla \mtx{x}\|_1$, while the \emph{isotropic} total variation seminorm is a mixed $\ell_1$-$\ell_2$ norm of the $d$-dimensional discrete gradient,
 
\begin{equation}\label{TV2}
\| \mtx{x} \|_{TV_2} \defby \sum_{\alpha\in[N]^d} \left(\sum_{\ell = 1}^d \big( \nabla \mtx{x}  \big)_{\alpha,\ell}^2\right)^{1/2}  \defby \| \nabla \mtx{x} \|_{1,2}.
\end{equation}

A linear operator ${\cal A}: \C^{N^d} \rightarrow \C^{r}$ can be regarded as a sequence of multidimensional arrays componentwise,
\begin{equation}\label{ip}
y_k = [{\cal A}(\mtx{x})]_k = \scalprod{ {\bm a}_k, \mtx{x} },
 \end{equation}
and a linear operator ${\cal A}: \C^{N^d} \rightarrow \C^{N^d}$ can be expressed similarly through its components $y_{\alpha} = [{\cal A}(\mtx{x})]_{\alpha} = \scalprod{ {\bm a}_{\alpha}, \mtx{x} }$.
If ${\cal A}:  \C^{N^d} \rightarrow \C^{r_1}$ and ${\cal B}: \C^{N^d} \rightarrow \C^{r_2}$ then the \emph{row direct sum operator} ${\cal M} = {\cal A} \oplus_r {\cal B}$ is the linear operator from $\C^{N^d}$ to $\C^{r_1 + r_2}$ with component arrays ${\cal M} = (\mtx{m}_k)_{k=1}^{r_1+r_2}$ given by
\begin{equation}
\label{rdirectsum}
\mtx{m}_k = \left\{ \begin{array}{ll} 
\mtx{a}_k,& 1 \leq k \leq r_1, \nonumber \\
\mtx{b}_{k-r_1}, & 1+r_1 \leq k \leq r_1 + r_2.
\end{array} \right.
\end{equation}
Alternatively, for linear operators ${\cal A}: \C^{N^d} \rightarrow \C^r$ and ${\cal B}: \C^{N^d} \rightarrow \C^r$, the  column direct sum operator ${\cal N} = {\cal A} \oplus_c {\cal B} : \C^{N^d \times 2} \rightarrow \C^{r}$ has component arrays ${\cal N} = (\mtx{n}_k)_{k=1}^{r}$ given by 
\begin{equation}
\label{cdirectsum}
(\mtx{n}_k)_{\alpha, \ell} = \left\{ \begin{array}{ll} 
(\mtx{a}_k)_{\alpha},& \ell=1, \nonumber \\
(\mtx{b}_k)_{\alpha}, & \ell=2.
\end{array} \right.
\end{equation}

{\subsection{The multidimensional Haar wavelet transform}~\label{backgrndwave}}
The Haar wavelet transform provides a sparsifying basis for natural signals such as images and movies, and is closely related to the discrete gradient.   For a comprehensive introduction to wavelets, we refer the reader to \cite{d}.

The (continuous) multidimensional Haar wavelet basis is derived from a tensor-product representation of the univariate Haar basis, which forms an orthonormal system for square-integrable functions on the unit interval and consists of the constant function 
$$
h^{0}(t) = \left\{ \begin{array}{ll} 1 & 0 \leq t < 1, \nonumber \\
0, & \textrm{otherwise},
\end{array} \right.
$$ 
the step function
$$
h^{1}(t) = \left\{ \begin{array}{ll} 1 & 0 \leq t < 1/2, \nonumber \\
-1 & 1/2 \leq t < 1,
\end{array} \right.
$$ 
and dyadic dilations and translations of the step function,
\begin{equation}
\label{haarsystem}
h_{j,k}(t) = 2^{j/2} h^{1}(2^j t - k); \quad j \in \N, \quad 0 \leq k < 2^j.
\end{equation}
The Haar basis for the higher dimensional space $L_2(Q)$ of square-integrable functions on the unit cube $Q = [0,1)^d$ consists of tensor-products of the univariate Haar wavelets.  Concretely, for $V=\{0,1\}^d - \{0\}^d$ and $\vct{e} = (e_1, e_2, \dots, e_d) \in V$, we define the multivariate functions $h^{\vct{e}}: L_2(Q) \rightarrow L_2(Q)$ by
$$
h^{\vct{e}}(u) = \prod_{e_i}h^{e_i}(u_i).
$$
The orthonormal Haar system on $L_2(Q)$ is then comprised of the constant function along with all functions of the form  
\begin{equation}
\label{bivariateH1}
h_{j,k}^{\vct{e}}(u) = 2^{jd/2} h^{\vct{e}}(2^j u - k), \quad \vct{e} \in V, \quad j \geq 1, \quad k \in \Z^d \cap 2^j Q. 
\end{equation}

The \emph{discrete} multidimensional Haar transform is derived from the continuous construction via the isometric identification between $\C^{N^d}$ and $\Sigma_N \subset L_2(Q)$: defining  
\begin{equation}
\label{bivariateH}
\mtx{h}_0(\alpha) = N^{-d/2}, \quad \quad \mtx{h}_{j,k,\vct{e}}(\alpha) = N^{-d/2}h_{j,k}^{\vct{e}}(\alpha/N), \quad \alpha \in [N]^d,
\end{equation}
the matrix product computing the discrete Haar transform can be expressed as ${\cal H}(\mtx{x}) = \big( \scalprod{ \mtx{h}_0 , \mtx{x} }, (\scalprod{\mtx{h}_{j,k,\vct{e}},\mtx{x}}) \big) $, where the indices $(j,k,e)$ are in the range \eqref{bivariateH1} but with $j \leq d-1$, which is a set of size $[N]^d - 1$.  Note that with this normalization, the transform is orthonormal.

\subsection{Gradient versus wavelet sparsity}
The following is a corollary of a remarkable result from Cohen, Dahmen, Daubechies, and DeVore \cite{cddd} which bounds the compressibility of a function's wavelet representation by its bounded variation, and will be very useful in our analysis.

 \begin{proposition}[Corollary of Theorem 1.1 from \cite{cddd}]
\label{cor:cdpx}
There is a universal constant $C > 0$ such that the following holds for any { mean-zero} $\mtx{x} \in \C^{N^d}$ in dimension $d \geq 2$: if the Haar transform coefficients $\mtx{c} = {\cal H}(\mtx{x})$ are partitioned by their support into blocks $\mtx{c}_{j,k} = ( \scalprod{ \mtx{h}_{j,k,e}, \mtx{x}} )_{e \in V}$ of cardinality $| \mtx{c}_{j,k} | = 2^d-1$, then the coefficient block of $k$th largest $\ell_2$-norm, denoted by $\mtx{c}_{(k)},$ has $\ell_2$-norm bounded by
$$
\| \mtx{c}_{(k)}\|_2 \leq C \frac{\| \mtx{x}  \|_{TV_1}}{k \cdot 2^{d/2-1}}
$$
Thus by equivalence of the anisotropic and isotropic total variation seminorms up to a factor of $\sqrt{d}$, 
$$
\| \mtx{c}_{(k)}\|_2 \leq C \frac{\sqrt{d} \| \mtx{x}  \|_{TV_2}}{k \cdot 2^{d/2-1}}.
$$
\end{proposition}

Proposition \ref{cor:cdpx}, whose derivation from Theorem 1.1 of ~\cite{cddd} is outlined in the appendix, will be crucial in our proofs of robust recovery via total variation.

\section{The main result}\label{sec:main}
Our main result concerns near-optimal recovery guarantees for multidimensional total variation minimization from compressed measurements.  Recall that a linear operator ${\cal A}: \C^{N^d} \rightarrow \C^r$ is said to have the restricted isometry property (RIP) of order $s$ and level $\delta \in (0,1)$  when
\begin{align}
\label{Operatorrip}
(1 - \delta) \| \mtx{x} \|_2^2 \leq \| {\cal{A}}( \mtx{x}) \|_2^2 \leq (1 + \delta) \| \mtx{x} \|_2^2 \hspace{12mm} \textrm{ for all}\hspace{2mm} s \textrm{-sparse } \mtx{x} \in \mathbb{C}^{N^d}.
\end{align}
A linear operator ${\cal A} = ({\bm a}_k): \C^{N^d} \rightarrow \C^r$ satisfies the RIP if and only if the $r \times N^d$ matrix $\mtx{A}$ whose $k$th row consists of the unraveled entries of the $k$th multidimensional array ${\bm a}_k$ satisfies the classical RIP, \eqref{def:rip}, and so without loss of generality we treat both definitions of the RIP as equivalent.

For our main result it will be convenient to define for a multidimensional array ${\bm a} \in \C^{N^{\ell-1}\times (N-1) \times N^{d-\ell}}$ the associated arrays  ${\bm a}_{0_{\ell}} \in \C^{N^d}$ and ${\bm a}^{0_{\ell}} \in \C^{N^d}$ obtained by concatenating a block of zeros to the beginning and end of ${\bm a}$ oriented in the $\ell$th direction:
\begin{equation}
(\vct{a}^{0_{\ell}})_{\alpha} = \left\{ \begin{array}{ll} 0, & \alpha_{\ell} = 1 \\
a_{\alpha_1,\ldots,\alpha_{\ell}-1,\ldots,\alpha_d}, & 2 \leq \alpha_{\ell} \leq N
\end{array}  \right.
\label{zeropad}
\end{equation}
and
\begin{equation}
(\vct{a}_{0_{\ell}})_{\alpha} = \left\{ \begin{array}{ll} 0, & \alpha_{\ell} = N \\
a_{\alpha_1,\ldots,\alpha_{\ell},\ldots,\alpha_d}, & 1 \leq \alpha_{\ell} \leq N - 1.
\end{array}  \right.
\label{zeropad2}
\end{equation}
The following lemma relating gradient measurements with ${\bm a}$ to signal measurements with ${\bm a}^{0_{\ell}}$ and ${\bm a}_{0_{\ell}}$ can be verified by direct algebraic manipulation and thus the proof is omitted.

\begin{lemma}
\label{padderiv}
Given $\mtx{x} \in \C^{N^d}$ and $\mtx{a} \in  \C^{N^{\ell-1}\times (N-1) \times N^{d-\ell}}$,
\begin{equation}
\label{padrelatex}
\scalprod{\mtx{a}, \mtx{x}_{r_{\ell}}} = \scalprod{\mtx{a}^{0_{\ell}}, \mtx{x}} - \scalprod{\mtx{a}_{0_{\ell}}, \mtx{x}}, \nonumber
\end{equation}
where the directional derivative ${\bm x}_{r_{\ell}}$ is defined in \eqref{eq:multidd}.  
\end{lemma} 
For a linear operator ${\cal A} = ({\bm a}_k): \C^{N^{\ell-1}\times (N-1) \times N^{d-\ell}} \rightarrow \C^{m}$ we define the operators 
${\cal A}^{0_{\ell}}: \C^{N^d} \rightarrow \C^{m}$ and ${\cal A}_{0_{\ell}}: \C^{N^d} \rightarrow \C^m$ as the sequences of arrays $({\bm a}^{0_k}_k)_{k=1}^m$ and $({{\bm a}_{0_k}}_k)_{k=1}^m$, respectively.  We conclude from Lemma \ref{padderiv} that ${\cal A}(\mtx{x}_{r_{\ell}}) = {\cal A}^{0_{\ell}}(\mtx{x}) - {\cal A}_{0_{\ell}}(\mtx{x})$.  

We are now prepared to state our main result which shows that total variation minimization yields stable recovery of $N^d$-dimensional signals from RIP measurements.

\begin{mainthm}\label{thm:gen}
Let $N = 2^n$.  Fix integers $p$ and $q$.  Let ${\cal A}: \C^{N^d} \rightarrow \C^{p}$ be such that, composed with the orthonormal Haar wavelet transform, ${\cal A} {\cal H}^{*} : \C^{N^d} \rightarrow \C^{p}$ has the restricted isometry property of order $2ds$ and level $\delta < 1$. 
Let ${\cal B}_{1}, {\cal B}_{2}, \dots, {\cal B}_{d}$ with ${\cal B}_{j}: \C^{N^{d-1}(N-1)} \rightarrow \C^{q}$ be such that ${\cal B} = {\cal B}_1 \oplus_c {\cal B}_2 \oplus_c \dots \oplus_c {\cal B}_d : \C^{N^{d-1}(N-1)} \rightarrow \C^{dq}$ has the restricted isometry property of order $5ds$ and level $\delta < 1/3$.
Set $m = 2dq + p$, and consider the linear operator ${\cal M}: \C^{N^d} \rightarrow \C^{m}$ given by
\begin{equation}
\label{measure:M}
{\cal M} = {\cal A} \oplus_r  \big[ {\cal B}_1\big]^{0_1} \oplus_r \big[{\cal B}_1\big]_{0_1}  \oplus_r \dots \oplus_r  \big[ {\cal B}_{\ell}\big]^{0_{\ell}} \oplus_r \big[{\cal B}_{\ell}\big]_{0_{\ell}} \oplus_r \dots \oplus_r \big[ {\cal B}_d\big]^{0_d} \oplus_r \big[{\cal B}_d\big]_{0_d}.
\end{equation}
The following holds for any $\mtx{x} \in \C^{N^d}$. From noisy measurements $\vct{y} = {\cal M}(\mtx{x}) + \vct{\xi}$ with noise level $\| \vct{\xi} \|_2 \leq \varepsilon$, the solution to
\begin{eqnarray}
\label{tv}
\mtx{\hat{x}} = \argmin_{\mtx{z}} \| \mtx{x} \|_{TV_2}  \quad  \textrm{such that}  \quad \| {\cal M}(\mtx{z}) - \vct{y} \|_2   \leq \varepsilon
\end{eqnarray}
satisfies:

\begin{enumerate}[i)]

\item $\| \nabla(\mtx{x}-\mtx{\hat{x}}) \|_2 \lesssim  \frac{ \| \nabla\mtx{x} - (\nabla\mtx{x})_S \|_{1,2}}{\sqrt{s}} + \sqrt{d}\varepsilon,$

\item $\| \mtx{x}-\mtx{\hat{x}} \|_{TV_2} \lesssim  \| \nabla\mtx{x} - (\nabla\mtx{x})_S \|_{1,2} + \sqrt{sd} \varepsilon,$

\item $\| \mtx{x}-\mtx{\hat{x}} \|_2 \lesssim  \log(N^d)\Big( \frac{\| \nabla\mtx{x} - (\nabla\mtx{x})_S \|_{1,2}}{\sqrt{s}}+\sqrt{d} \varepsilon\Big)$,
\end{enumerate}
\vspace{3mm}
where $\| \mtx{z} \|_{TV_2} = \| \nabla \mtx{z} \|_{1,2}$ is the isotropic total variation seminorm, $\| \mtx{x} \|_{1,2} = \sum_{\alpha \in [N]^d} \big( \sum_{\ell=1}^d x_{\alpha, \ell}^2  \big)^{1/2}$ is the associated mixed $\ell_1-\ell_2$ norm, and $(\nabla\mtx{x})_S$ is the signal gradient $\nabla\mtx{x}$  restricted to the subset $S = \big\{ \{ \alpha_{(1)}, \dots, \alpha_{(s)} \}  \times [d] \big\}$ of  $s$ largest-magnitude block norms $\| (\nabla \mtx{x})_{\alpha_{(k)}} \|_2 $. 
\end{mainthm}

\noindent {\bf Remarks.}

{\bf 1.} Following the same lines of reasoning as the proof of this Main Theorem, one may derive error bounds for \emph{anisotropic} TV minimization that are a bit tighter, namely, one only requires RIP of order $5s$.  We believe that the suboptimal bounds for isotropic TV are merely an artifact of our proof technique, as isotropic TV minimization is preferred in practice. 

{\bf 2.} The third bound shows that the reconstruction error is proportional (up to a logarithmic factor) to the noise level $\varepsilon$ and the \textit{tail of the gradient} of the signal $\mtx{x}$.  A number of $m \approx sd\log(N^d)$ i.i.d. and properly normalized Gaussian measurements can be used to construct the measurement operator ${\cal M}$ which, with high probability, satisfies the required RIP conditions of the theorem \cite{badadewa08,RV08:sparse}.  From this number $m$ of measurements, the error guarantees are optimal up to the factor of $d$ required RIP measurements, factor of $\sqrt{d}$ on the noise dependence, and logarithmic factor in the signal dimension $N^d$.  We emphasize here that the specific construction of the measurement ensemble is likely only an artifact of the proof, and that more general RIP measurements are likely possible.  See also~\cite{krahmer2012beyond} for results using Fourier measurements (for $d=2$).

{\bf 3.} The main theorem recovers the total variation guarantees of~\cite{nwtv} when $d=2$ up to a $\log(1/s)$ term.  This term is lost only because in the higher-dimensional analysis, our proofs require blocking of the wavelet coefficients.  This term can be recovered by applying a more efficient blocking strategy and by writing in terms of $p$ in~\eqref{tail1} of the proof.  We write the bound as-is for simplicity.

{\bf 4.} The requirement of sidelength $N = 2^n$ is not an actual restriction, as signals with arbitrary side-length $N$ can be extended via reflections across each dimension to a signal of side-length $N = 2^n$ without increasing the total variation by more than a factor of $2^d$.  This requirement again seems to be only an artifact of the proof and one need not perform such changes in practice. 

\vspace{5mm}

We now turn to the proof of the main theorem.  We will first prove the gradient-level recovery bounds $(i)$ and $(ii)$, and then use these to prove the signal-level recovery bound $(iii)$ via Sobolev inequalities for incoherent subspaces.  Along the way, we must take care to balance estimates involving the gradient vectors $(\nabla \mtx{x}_{\alpha,\ell})_{\ell=1}^d \in \C^{d}$ and their block norms $\| \nabla \mtx{x}_{\alpha} \|_2$, as well as the blocks of wavelet coefficients ${\bf c}_{(k)} \in \C^{2^d-1}$ associated to a dyadic cube and their block norms $\| {\bf c}_{(k)} \|_2$. 

\section{Stable gradient recovery}\label{stabgrad}

In this section we prove statements $(i)$ and $(ii)$ of the main theorem concerning stable gradient recovery, using standard results in compressed sensing combined with a summation by parts  trick provided by the specific form of the measurements in the Main Theorem and Lemma \ref{padderiv}.

Recall that when a signal obeys a tube and cone constraint we can bound the norm of the entire signal, as in~\cite{crt}.  We refer the reader to Section A.1 of~\cite{nwtv} for a complete proof.

\begin{proposition}
\label{cone-tube}
Suppose that ${\cal B}$ is a linear operator satisfying the restricted isometry property of order $5ds$ and level $\delta < 1/3$, and suppose that the signal ${\bm h}$ satisfies a tube constraint
$$
\| {\cal B}({\bm h}) \|_2 \leq \sqrt{2d}\varepsilon.
$$
Suppose further that using the notation of the Main Theorem, for a subset $R = R'\times [d]$ of cardinality $|R| \leq sd$ (meaning $|R'| \leq s$), ${\bm h}$ satisfies a cone-constraint
\begin{equation}
\label{cc}
\| {\bm h}_{R^c} \|_{1,2} \leq \| {\bm h}_R \|_{1,2} + {\sigma}.
\end{equation}
Then 
\begin{equation}\label{eq:h2}
\| {\bm h} \|_2 \lesssim \frac{{\sigma}}{\sqrt{s}} + \sqrt{d}\varepsilon
\end{equation}
 and
\begin{equation}\label{eq:h1}
\| {\bm h} \|_{1,2} \lesssim \sigma + \sqrt{sd} \varepsilon.
\end{equation}
\end{proposition}

Proposition~\ref{cone-tube} generalizes results in~\cite{CRT06:Stable} and its proof is included in the appendix.  Using Proposition~\ref{cone-tube} and RIP assumptions on the operator ${\cal B}$, the gradient-level recovery guarantees (i) and (ii) reduce to proving that the discrete gradient of the residual signal error satisfies the tube and cone constraints.

\begin{proof}(Main Theorem, statements (i) and (ii).) \\
Let  ${\bm v} = \mtx{x} - \mtx{\hat{x}}$ be the residual error, and set $\mtx{h} = \nabla \mtx{v} = \nabla \mtx{x} - \nabla \mtx{\hat{x}} \in \C^{N^d \times d}$. \\
Then we have
\begin{description}
\item[]{\bfseries Cone Constraint.}  Consider the block norm $\| ( \nabla {\bf x} )_\alpha \|_2 = \big( \sum_{\ell=1}^d x_{\alpha,\ell}^2 \big)^{1/2}$ associated to the index $\alpha \in [N]^d$, and denote by $\alpha_{(j)}$ the index of the $j$th largest block norm $\| ( \nabla {\bf x} )_{\alpha_{(j)}} \|_2$, and let $S = \{ \alpha_{(1)}, \dots, \alpha_{(s)} \}  \times [d] $.   Since $\mtx{\hat{x}} = \mtx{x} - {\bm v}$ is a minimizer of (TV) and $\mtx{x}$ satisfies the feasibility constraint in (TV), we have that $\| \nabla \mtx{\hat{x}} \|_{1,2} \leq \| \nabla \mtx{x} \|_{1,2}$.  By the reverse triangle inequality,
\begin{align*}
\| (\nabla\mtx{x})_S \|_{1,2} - \| {\bm h}_S\|_{1,2} - \| (\nabla\mtx{x})_{S^c}\|_{1,2} &+ \| {\bm h}_{S^c}\|_{1,2} \\
&\leq \| (\nabla\mtx{x})_S - {\bm h}_S\|_{1,2} + \| (\nabla\mtx{x})_{S^c} - {\bm h}_{S^c}\|_{1,2}\\
&= \|\nabla\mtx{\hat{x}}\|_{1,2}\\
&\leq \|\nabla\mtx{x}\|_{1,2}\\
&= \| (\nabla\mtx{x})_S\|_{1,2} + \| (\nabla\mtx{x})_{S^c}\|_{1,2}.
\end{align*}
This yields the cone constraint

$$
\| {\bm h}_{S^c} \|_{1,2} \leq \| {\bm h}_S \|_{1,2} + 2\| (\nabla\mtx{x})_{S^c} \|_{1,2}
$$
\item[] {\bfseries Tube constraint.} Recall that $\mtx{v} = \mtx{x} - \mtx{\hat{x}}$. Since both $\mtx{x}$ and $\mtx{\hat{x}}$ are feasible solutions to (TV), Jensen's inequality gives 
\begin{equation}
\| {\cal M}({\bm v}) \|_2^2 \leq 2\| {\cal M}(\mtx{x}) - \vct{y} \|_2^2 + 2\| {\cal M}(\mtx{\hat{x}}) - \vct{y} \|_2^2\leq 4 \varepsilon^2 \nonumber
\end{equation}

By Lemma \ref{padderiv}, we have for each component operator ${\cal B}_j$,
\begin{eqnarray}
{\cal B}_j({\bm v}_{r_j}) &=& [{\cal B}_j]^{0_j} ({\bm v}) - [{\cal B}_j]_{0_j} ({\bm v}) 
\end{eqnarray}
Then ${\cal B}(\nabla\mtx{v}) = \sum_{j=1}^d {\cal B}_j({\bm v}_{r_j}),$ (where we assume that $\nabla \mtx{v}$ is ordered appropriately) and
\begin{eqnarray}
\| {\cal B}(\nabla\mtx{v}) \|_2^2 &=&\|\sum_{j=1}^d {\cal B}_j({\bm v}_{r_j}) \|_2^2 \nonumber \\
&\leq& d \sum_{j=1}^d \| {\cal B}_j({\bm v}_{r_j}) \|_2^2  \nonumber \\
&\leq& 2d \sum_{j=1}^d \Big( \| {\cal B}_j]^{0_j} ({\bm v}) \|_2^2 + \| {\cal B}_j]_{0_j} ({\bm v}) \|_2^2 \Big)  \nonumber \\
 &\leq& 2d\| {\cal M}({\bm v}) \|_2^2 
 \nonumber \\
 &\leq& 8d\varepsilon^2.
\end{eqnarray}
\end{description}

In light of Proposition~\ref{cone-tube} the proof is complete.

\end{proof}

\begin{remark}
The component operator ${\cal A}$ from the main theorem was not used at all in deriving properties $(i)$ and $(ii)$; on the other hand, \emph{only} the measurements in ${\cal A}$ will be used to derive property $(iii)$ from $(i)$ and $(ii)$.  We conjecture that all measurements in the main result apart from those in the component operator ${\cal A}$ are artifacts of the proof techniques herein.
\end{remark}

\section{A Sobolev inequality for incoherent subspaces}\label{sec:sob}

We now derive a strengthened Sobolev inequality for signals lying near the null space of a matrix which is incoherent to the Haar wavelet basis.

\begin{theorem}[Sobolev inequality for incoherent subspaces]
\label{strongsobo}
Let $d \geq 2$ and let $N = 2^n$.  Let ${\cal A}: \C^{N^d} \rightarrow \C^m$ be a linear map such that, composed with the multivariate Haar wavelet transform ${\cal H}: \C^{N^d} \rightarrow \C^{N^d}$, the resulting operator ${\cal A} {\cal H}^{*}: \C^{N^d} \rightarrow \C^{m}$ satisfies the restricted isometry property of order $2s$ and level $\delta < 1$.  Then there is a universal constant $C > 0$ such that the following holds: if $\mtx{v} \in \C^{N^d}$ satisfies the tube constraint $\| {\cal A}(\mtx{v}) \|_2 \leq \varepsilon$, then
\begin{equation}
\label{sobstrong1}
\| \mtx{v} \|_2 \leq C\frac{\| \mtx{v} \|_{TV_1}}{\sqrt{s}} \log(N^d) + \varepsilon. 
\end{equation}
and so, by equivalence of the isotropic and anisotropic total variation seminorms up to a factor of $\sqrt{d}$, 
\begin{equation}
\label{sobstrong12}
\| \mtx{v} \|_2 \leq C\frac{\| \mtx{v} \|_{TV_2}}{\sqrt{s/d}} \log(N^d) + \varepsilon. 
\end{equation}
\end{theorem}

\begin{remark}
The RIP assumptions on ${\cal A}{\cal H}^* = {\cal A}{\cal H}^{-1}$ imply that for $\mtx{v}$ with $s$-sparse wavelet representation, 
$$ 
\| {\cal A}(\vct{v}) \|_2 =  \| {\cal A}{\cal H}^*{\cal H}(\vct{v}) \|_2 \approx (1\pm \delta) \| {\cal H} (\vct{v}) \|_2 \approx (1\pm \delta) \| \vct{v} \|_2,$$
with the final equality holding because ${\cal H}$ is unitary.  This implies that the null space of ${\cal A}$ cannot contain any signals admitting an $s$-sparse wavelet expansion, apart from the zero vector.  In particular, the null space of ${\cal A}$ contains no nontrivial constant signals and thus intersects the null space of $\nabla$ trivially.
\end{remark}

Theorem \ref{strongsobo} admits corollaries for various families of random matrices with restricted isometries.  For Gaussian random matrices, the theorem implies the following.  
\begin{cor}
\label{sscor}
Let ${\cal A}: \C^{N^d} \rightarrow \C^m$ be a linear map realizable as an $N^d \times m$ matrix whose entries are mean-zero i.i.d. Gaussian random variables.  Then there is a universal constant $c > 0$ such that with probability exceeding $1 - e^{-cm}$, the following bound holds for any $\mtx{x} \in \C^{N^d}$ lying in the null-space of ${\cal A}$: 
\begin{equation}
\| \mtx{x} \|_2 \lesssim \frac{\| \mtx{x} \|_{TV_1}}{\sqrt{m}} [\log(N^d)]^2.
\end{equation}
and
\begin{equation}
\| \mtx{x} \|_2 \lesssim \frac{\sqrt{d} \| \mtx{x} \|_{TV_2}}{\sqrt{m}}[\log(N^d)]^2.
\end{equation}

\end{cor}
{ 
\begin{proof}[Proof of Corollary \ref{sscor}]
From results on Gaussian matrices and the restricted isometry property (see e.g.~\cite{CT05:Decoding,MPJ06:Uniform,badadewa08,RV08:sparse}), ${\cal A}$ satisfies the RIP of order $2s$ and level $\delta < 1$ with probability exceeding $1 - e^{-cm}$ when $s$ is proportional to $m / \log(N^d)$.  Substituting this value for $s$ into~\eqref{sobstrong1} and~\eqref{sobstrong12} yields the claim.
\end{proof}

\begin{proof}[Proof of Theorem~\ref{strongsobo}]
Consider the signal error $\mtx{v} = \mtx{x} - \mtx{\hat{x}}$, and consider its orthogonal decomposition $\mtx{v} = \mtx{v^0} + \mtx{v^1}$ where $\mtx{v^0} = \scalprod{ \mtx{h}_0, \mtx{v} }\mtx{h}_0$ is constant and $\mtx{v^1} = \mtx{v} - \mtx{v^0}$ is mean-zero (recall that $\mtx{h}_0$ has constant entries equal to $N^{-d/2}$).
Let $\mtx{c} = {\cal H}({\mtx{v}}) \in \C^{N^d}$ represent the orthonormal Haar transform of $\mtx{v}$.  Suppose without loss of generality that the desired sparsity level $s$ is either smaller than $2^d-1$ or a positive multiple of $2^d-1$, and write $s = p(2^d-1)$ where either $p \in \mathbb{N}$ or $p \in (0,1)$ (for arbitrary $s \in \mathbb{N}$, we could consider $s' = \lceil{ s/(2^d-1) \rceil}$ which satisfies $s' \leq 2s$). 
 
 Let $S = S_0 \subset [N]^d$ be the index set of cardinality $s$ which includes the constant Haar coefficient $c_0 = \scalprod{ \mtx{h}_0, \mtx{v} }$ along with the $s-1$ largest-magnitude entries of $\mtx{c}$ (not including $c_0$).   Let $S_1$ be the set of $s$ largest-magnitude entries of $\mtx{c}$ in  $[N]^d \setminus S_0$, and so on.  Note that $\mtx{c}_{S}$ and similar expressions above and below can have both the meaning of restricting $\mtx{c}$ to the indices in $S$ as well as being the array whose entries are set to zero outside $S$.  

Now, since $\mtx{v^1}$ is mean-zero and $\mtx{c}_{S_0^c} = {\cal H}({\mtx{v^1}})_{S_0^c}$, we may apply Proposition \ref{cor:cdpx}.  To that end, consider the decomposition of $\mtx{c}$ into blocks $\mtx{c}_{(k)}$ of cardinality $2^d-1$ as in Proposition \ref{cor:cdpx}, according to the support of their wavelets.  By definition, $\| \mtx{c}_{S_0} \|_1$ is at least as large as $\| \mtx{c}_{\Omega} \|_1$ for any other $\Omega \subset [N]^d$ of cardinality $s-1$.  Consequently, $\| \mtx{c}_{S_0^c} \|_1$ is \emph{smaller} than $\| \mtx{c}_{\Omega^c} \|_1$ for any other $\Omega \subset [N]^d$ of cardinality $s-1$.  Thus, because $s = p(2^d-1)$, 
\begin{eqnarray}
\label{tail1}
\| \mtx{c}_{S_0^c} \|_1 
&\leq& \sum_{j \geq p+1}  \|\mtx{c}_{(j)}\|_1 \nonumber \\
&\leq& (2^{d}-1)^{1/2} \sum_{j \geq p+1}  \|\mtx{c}_{(j)} \|_2 \nonumber \\
&\lesssim& \| \mtx{v} \|_{TV_1}  \sum_{\ell = p+1}^{N^d}  \frac{1}{\ell}  \nonumber \\
&\lesssim& \| \mtx{v}  \|_{TV_1} \log(N^d),
\end{eqnarray}
where the second to last inequality follows from Proposition \ref{cor:cdpx} and the last inequality from properties of the geometric summation.

We use a similar procedure to bound the $\ell_2$-norm of the residual,
\begin{eqnarray}
\label{tail2}
\| \mtx{c}_{S_0^c} \|_2^2 &\lesssim& \sum_{j\geq p+1} \|\mtx{c}_{(j)} \|_2^2 \nonumber \\
 &\lesssim& \frac{\|{\bm v} \|_{TV_1}^2}{2^d} \sum_{\ell = p+1}^{N^d} \frac{1}{\ell^2} \nonumber \\
 &\lesssim& \frac{(\| {\bm v} \|_{TV_1})^2}{2^d \max{(1,p)}}\nonumber \\
 &\lesssim& \frac{(\| {\bm v} \|_{TV_1})^2}{s}.
\end{eqnarray}
Then, $\| \mtx{c}_{S_0^c} \|_2 \lesssim\| {\bm v} \|_{TV_1} / \sqrt{s}$.

By assumption, ${\bm v}$ satisfies the tube constraint $\| {\cal A}({\bm v}) \|_2 \leq \varepsilon$ and ${\cal A} {\cal H}^{*} = {\cal A} {\cal H}^{-1}$ satisfies the restricted isometry property of order $5ds$.  We conclude that
\begin{eqnarray}
\varepsilon &\geq& \| {\cal A}({\bm v}) \|_2 = \| {\cal A}{\cal H}^{*}(\mtx{c}) \|_2 \nonumber \\
&\geq& \| {\cal A} {\cal H}^{*} (\mtx{c}_{S_0} + \mtx{c}_{S_1}) \|_2 - \sum_{k=2}^r  \| {\cal A}{\cal H}^{*} (\mtx{c}_{S_k}) \|_2 \nonumber \\
&\geq& (1 - \delta) \| \mtx{c}_{S_0} + \mtx{c}_{S_1} \|_2 -(1 + \delta)  \sum_{k=2}^r \| \mtx{c}_{S_k} \|_2 \nonumber \\
&\geq&  (1- \delta) \| \mtx{c}_{S_0} \|_2 - (1 + \delta) \frac{1}{\sqrt{s}} \| \mtx{c}_{S_0^c} \|_1,
\end{eqnarray}
the last inequality holding because the magnitude of each entry in the array $\mtx{c}_{S_k}$ is smaller than the average magnitude of the entries in the array $\mtx{c}_{S_{k-1}}$.
Along with the tail bound \eqref{tail1}, we can then conclude that, up to a constant in the restricted isometry level $\delta$,  
\begin{eqnarray}
\| \mtx{c}_{S_0} \|_2 &\lesssim& \varepsilon + \log(N^d)\Big( \frac{\| {\bm v} \|_{TV_1}}{\sqrt{s}}\Big).
\end{eqnarray}
Combining this bound with the $\ell_2$-tail bound \eqref{tail2} and recalling that the Haar transform is unitary, we find that
\begin{equation}
\| {\bm v} \|_2 = \| {\cal H}^{*} \mtx{c} \|_2 =  \| \mtx{c} \|_2  \leq \| \mtx{c}_{S_0} \|_2 + \| \mtx{c}_{S_0^c} \|_2  \lesssim \varepsilon + \log(N^d)\Big( \frac{\| {\bm v} \|_{TV_1}}{\sqrt{s}}\Big),
\end{equation}
which completes the proof.

\end{proof}
}

\subsection{Proof of the Main Theorem}

Because we proved the bounds (i) and (ii) from the main theorem concerning gradient-level recovery bounds in Section~\ref{stabgrad}, it remains only to prove the signal recovery error bound (iii).

By feasibility of both $\mtx{x}$ and $\mtx{\hat{x}}$ for the constraint in the total variation minimization program, the signal error ${\bm v} = \mtx{x} - \mtx{\hat{x}}$ obeys the tube-constraint $\| {\cal A}({\bm v}) \|_2 \leq 2\varepsilon$.  Applying Theorem \ref{strongsobo} with RIP of order $2sd$ (in place of $2s$) and the total variation bound (ii) yields
\begin{align*}
\| \mtx{x}-\mtx{\hat{x}} \|_2 &= \| {\bm v} \|_2 \\
 &\lesssim \varepsilon + \log(N^d)\Big(\frac{\| {\bm v} \|_{TV_1}}{\sqrt{ds}}\Big)\\
  &\lesssim \varepsilon + \log(N^d)\Big(\frac{\| {\bm v} \|_{TV_2}}{\sqrt{s}}\Big)\\
 &\lesssim \varepsilon + \frac{ \log(N^d)}{\sqrt{s}} \left( \| \nabla\mtx{x} - (\nabla\mtx{x})_S \|_{1,2} + \sqrt{sd} \varepsilon \right) \\
 &\lesssim \log(N^d)\left(\sqrt{d}\varepsilon + \frac{\| \nabla\mtx{x} - (\nabla\mtx{x})_S \|_{1,2}}{\sqrt{s}}\right).
\end{align*}
 The proof completes. 
   
\appendix

\subsection{Derivation of Proposition \ref{cor:cdpx}}
Recall that the space $L_p(\Omega)$ $(1 \leq p < \infty)$ for $\Omega  \subset \R^d$ consists of all functions $f$ satisfying
$$
\| f \|_{L_p(\Omega)} = \Big( \int_{\Omega} | f(u) |^p du \Big)^{1/p} < \infty.
$$

The space BV$(\Omega)$ of functions of bounded variation over the unit cube $Q = [0,1)^d$ is often used as a continuous model for natural images.  Recall that a function $f \in L_1(Q)$ has finite bounded variation if and only if its distributional gradient is a bounded Radon measure, and this  measure generates the BV seminorm $| f |_{BV(\Omega)}$.   More precisely,  
 \begin{definition}
For a vector $\vct{v} \in \mathbb{R}^d$, we define the difference operator $\Delta_{\vct{v}}$ in the direction of $\vct{v}$ by 
$$
\Delta_{\vct{v}}(f,\mtx{x}) := f(\mtx{x} + \vct{v}) - f(\mtx{x}).
$$ 
We say that a function $f \in L_1(Q)$ is in $BV(Q)$ if and only if
$$
V_Q(f) \defby \sup_{h > 0} h^{-1} \sum_{j=1}^d \| \Delta_{h \vct{e}_j}(f,\cdot) \|_{L_1(Q(h\vct{e}_j))} =\lim_{h \rightarrow 0}h^{-1} \sum_{j=1}^d  \| \Delta_{h \vct{e}_j} (f, \cdot) \|_{L_1(Q(h \vct{e}_j))} < \infty
$$
where $\vct{e}_j$ denotes the $j$th coordinate vector.  The function $V_Q(f)$ provides a seminorm for $BV(Q)$:
$$
| f |_{BV(Q)} \defby V_Q(f).
$$
\end{definition}
In particular, piecewise constant functions are in the space $BV(Q)$, and we have
\begin{lemma}\label{bvtv}
Let $N = 2^n$.  Let $\mtx{x} \in \C^{N^d}$ and let $f \in \Sigma^d_N$ be its isometric embedding as a piecewise constant function.   Then $| f |_{BV} \leq N^{-d/2 + 1} \| \mtx{x} \|_{TV_1}$.
\end{lemma}

\begin{proof}
For ${h} < \frac{1}{N}$, 
\begin{equation}
\label{delta1}
\Delta_{h \vct{e_k}} \big( f, \vct{u} \big) = \left\{ \begin{array}{ll} 
N^{d/2} (\mtx{x}_{\vct{{\ell}^{(k)}}} - \mtx{x}_{\vct{\ell}} ) & \frac{\ell_i}{N} - h \leq u_i \leq \frac{\ell_i}{N}, \nonumber \\
0, & \textrm{else}, 
\end{array} \right . 
\end{equation}
where 

\begin{equation}
\vct{{\ell}}^{(k)}_i = \left\{ \begin{array}{ll} 
\vct{\ell}_i  & i\ne k, \nonumber \\
\vct{\ell}_i + 1, & i = k. 
\end{array} \right .
\end{equation}
Thus
\begin{eqnarray}
| f |_{BV} &=& \lim_{h \rightarrow 0} \frac{1}{h} \sum_{k=1}^d\left[  \int_{0}^1 \int_{0}^1 \ldots \int_{0}^1 | f(\vct{u}+h\vct{e_k}) - f(\vct{u}) | \hspace{1mm} d\vct{u}   \right]  \nonumber\\
&=& \sum_{k=1}^d N^{d/2}\left[\sum_{\vct{\ell}} \frac{1}{N^{d-1}}| \mtx{x}_{\vct{\ell}^{(k)}}-\mtx{x}_{\vct{\ell}} | \right]  \nonumber\\
&\leq& N^{-d/2+1}\|\nabla\vct{x}\|_1 = N^{-d/2+1}\| \mtx{x} \|_{TV_1}. \nonumber
\end{eqnarray}
\end{proof}

Cohen, Dahmen, Daubechies, and DeVore showed in \cite{cddd} that the properly normalized sequence of rearranged wavelet coefficients associated to a function $f \in L_2(\Omega)$ of bounded variation is in \emph{weak-$\ell_1$}, and its weak-$\ell_1$ seminorm is bounded by the function BV seminorm.   Using different normalizations to those used in \cite{cddd} --- we use the $L_2$-normalization for the Haar wavelets as opposed to the $L_1$-normalization --- we consider the Haar wavelet coefficients $f^{e}_{I} = \scalprod{f, h_I^e}$ and consider the wavelet coefficient block $f_I = (f_I^e)_{e \in E} \in \C^{2^d-1}$ associated to those Haar wavelets supported on the dyadic cube $I$.  With this notation, Theorem 1.1 of \cite{cddd} applied to the Haar wavelet system over $L_2(Q)$ reads:

\begin{proposition}
\label{haardecaybv}
Let $d \geq 2$. Then there exists a constant $C > 0$ such that the following holds for all mean-zero $f \in BV(Q)$.  Let the wavelet coefficient block with $k$th largest $\ell_2$-norm be denoted by $f_{(k)}$, and suppose that this block is associated to the dyadic cube $I_{j,k}$ with side-length $2^{-j}$.  Then 
$$
\|f_{(k)}\|_2 \leq C \frac{ 2^{j(d-2)/2} |f |_{BV} }{k}.
$$
\end{proposition}
Proposition \ref{cor:cdpx} results by translating Proposition \ref{haardecaybv} to the discrete setting of $\C^{N^d}$ and appealing to Lemma \ref{bvtv}.  We note that a stronger version of this result was provided for the 2-dimensional Haar wavelet basis in \cite{cdpx} and used in the proofs in \cite{nwtv}. 

\subsection{Proof of Proposition~\ref{cone-tube}}

Assume that the cone and tube constraints are in force.  Let $R_1 = R_1'\times [d] \subset R^c$ contain the $4s$ largest blocks of $\vct{h}$ on $R^c$, let $R_2 = R_2' \times [d]$ contain the next $4s$ largest, and so on.  We write $\vct{h}_{R_j}$ to mean the array $\vct{h}$ restricted to its elements indexed by $R_j$, and write $\vct{h}_{\alpha}$ to denote the array $({h}_{(\alpha,\ell)})_{\ell=1}^d$ at pixel $\alpha$.  Thus for any $\alpha \in R'_{j+1}$,
$$
\|\vct{h}_{\alpha}\|_2 \leq \frac{1}{4s} \sum_{{\beta}\in R'_j} \|\vct{h}_{{\beta}}\|_2 = \frac{1}{4s}\|\vct{h}_{R_j}\|_{1,2}.
$$

Therefore we have
$$
\|\vct{h}_{R_{j+1}}\|_2^2 = \sum_{\alpha\in R'_{j+1}}\|\vct{h}_{\alpha}\|_2^2 \leq \sum_{\alpha\in R'_{j+1}}\frac{1}{(4s)^2}\|\vct{h}_{R_j}\|_{1,2}^2 = \frac{1}{4s}\|\vct{h}_{R_{j}}\|_{1,2}^2.
$$

Combining this with the cone constraint yields

\begin{align*}
\sum_{j\geq 2}\|\vct{h}_{R_{j}}\|_2 \leq \frac{1}{\sqrt{4s}}\sum_{j\geq 1}\|\vct{h}_{R_{j}}\|_{1,2} &= \frac{1}{\sqrt{4s}}\sum_{\alpha\in (R')^c}\|\vct{h}_{\alpha}\|_2 \\
= \frac{1}{\sqrt{4s}}\|\vct{h}_{R^c}\|_{1,2} \leq \frac{1}{\sqrt{4s}}\|\vct{h}_{R}\|_{1,2} + \frac{1}{\sqrt{4s}}\sigma &\leq \frac{1}{2}\|\vct{h}_{R}\|_{2} + \frac{1}{\sqrt{4s}}\sigma,
\end{align*}

where in the last line we have utilized the fact that $\|\vct{h}_{R}\|_{1,2} \leq \sqrt{s}\|\vct{h}_{R}\|_{2}$.  Next, the tube constraint gives

\begin{align*}
\sqrt{2d}\varepsilon \geq \|\mathcal{B}(\vct{h})\|_2 &\geq \sqrt{1-\delta}\|\vct{h}_R + \vct{h}_{R_1}\|_2 - \sqrt{1+\delta}\sum_{j\geq 2}\|\vct{h}_{R_{j}}\|_2\\
&\geq \sqrt{1-\delta}\|\vct{h}_R + \vct{h}_{R_1}\|_2 - \sqrt{1+\delta}\left(\frac{1}{2}\|\vct{h}_{R}\|_{2} + \frac{1}{\sqrt{4s}}\sigma \right) \\
&\geq \left(\sqrt{1-\delta} - \frac{1}{2}\sqrt{1+\delta}\right)\|\vct{h}_R + \vct{h}_{R_1}\|_2 - \sqrt{1+\delta}\left(\frac{1}{\sqrt{4s}}\sigma\right).
\end{align*}

Using the fact that $\delta < 1/3$, this implies that
$$
\|\vct{h}_R + \vct{h}_{R_1}\|_2 \leq 10\sqrt{d}\varepsilon + \frac{3}{\sqrt{s}}\sigma.
$$

The bound~\eqref{eq:h2} then follows since

\begin{align*}
\|\vct{h}\|_2 &\leq \|\vct{h}_R + \vct{h}_{R_1}\|_2 + \sum_{j\geq 2}\|\vct{h}_{R_{j}}\|_2 \\
&\leq \|\vct{h}_R + \vct{h}_{R_1}\|_2 + \frac{1}{2}\|\vct{h}_{R}\|_{2} + \frac{1}{\sqrt{4s}}\sigma \\
&\lesssim  \sqrt{d}\varepsilon + \frac{1}{\sqrt{s}}\sigma.
\end{align*}  

Similarly, the bound~\eqref{eq:h1} follows from the cone constraint,
\begin{align*}
\|\vct{h}\|_{1,2} &\leq 2\|\vct{h}_R\|_{1,2} + \sigma\\
&\leq 2\sqrt{s}\|\vct{h}_R\|_{2}+ \sigma\\
&\lesssim \sqrt{s}(\sqrt{d}\varepsilon + \frac{1}{\sqrt{s}}\sigma) + \sigma\\
&= \sqrt{sd}\varepsilon + \sigma,
\end{align*} 

which completes the proof.

\subsection*{Acknowledgment}
We would like to thank John Doyle, Christina Frederick, and Mark Tygert for helpful improvements and insights.  We would also like to thank the reviewers of the manuscript for their thoughtful suggestions which significantly improved the manuscript.

\bibliographystyle{plain}
\bibliography{haar}

\end{document}